\documentclass[12pt]{article}
\usepackage{amsmath,amssymb,amsthm,latexsym,amscd,stmaryrd,mathrsfs,longtable,color}
\newtheoremstyle{standard}%
{9pt}%
{9pt}%
{\it}
{}%
{\bfseries}%
{}
{ }%
{#3}%

\newcommand{\db}[1]{(\!({#1})\!)}

\numberwithin{equation}{section}

\newcommand{\N}{{\mathbb N}}
\newcommand{\Z}{{\mathbb Z}}
\newcommand{\Q}{{\mathbb Q}}

\newcommand{\C}{{\mathbb C}}

\newcommand{\wi}{i}
\newcommand{\wj}{j}
\newcommand{\wk}{k}

\newcommand{\wx}{x}

\newcommand{\ws}{s}
\newcommand{\wm}{m}
\newcommand{\wn}{n}
\newcommand{\hei}{H}
\newcommand{\module}{M}

\newcommand{\sN}{{\mathscr N}}

\newcommand{\sU}{{\mathscr U}}

\newcommand{\wpp}{p}
\newcommand{\wq}{q}

\newcommand{\mmbox}[1]{\mbox{\scriptsize #1}}

\newcommand{\nor}{\begin{subarray}{c}\circ\\\circ\end{subarray}}

\newcommand{\fg}{{\mathfrak g}}

\newcommand{\ul}[1]{{#1}}
\newcommand{\lw}{w}
\newcommand{\ulw}{v}
\newcommand{\sv}{P}
\newcommand{\vac}{{\mathbf 1}}
\newcommand{\wot}{t}

\newcommand{\wl}{l}

\DeclareMathOperator{\wt}{wt}

\newtheorem{lemma}{Lemma}[section]
\newtheorem{theorem}[lemma]{Theorem}

\theoremstyle{definition}

\newtheorem{remark}[lemma]{Remark}

\theoremstyle{standard}

\title{Simple weak modules for the
fixed point subalgebra of the Heisenberg vertex operator algebra of rank $1$ by an automorphism of order $2$ and Whittaker vectors}
\author{Kenichiro Tanabe\footnote{Research was partially supported by the Grant-in-aid
(No. 15K04770) for Scientific Research, JSPS.}\\\\
Department of Mathematics\\
Hokkaido University\\
Kita 10, Nishi 8, Kita-Ku, Sapporo, Hokkaido, 060-0810\\
Japan\\\\
ktanabe@math.sci.hokudai.ac.jp}
\date{}
\begin{document}
\maketitle

\begin{abstract}
Let $M(1)$ be the vertex operator algebra with the Virasoro element $\omega$ associated to
the Heisenberg algebra of rank $1$ and let $M(1)^{+}$ be the subalgebra of $M(1)$ consisting of the fixed points of an automorphism of $M(1)$ of order $2$. 
We classify the simple weak $M(1)^{+}$-modules with a non-zero element $\lw$ such that
for some integer $\ws\geq 2$,  $\omega_i\lw\in\C\lw$ ($i=\lfloor \ws/2\rfloor+1,\lfloor \ws/2\rfloor+2,\ldots,\ws-1$), 
$\omega_{\ws}\lw\in\C^{\times}\lw$, and $\omega_i\lw=0$ for all $i>\ws$. 
The result says that
any such simple weak $M(1)^{+}$-module 
is isomorphic to some simple weak $M(1)$-module or to some $\theta$-twisted simple weak $M(1)$-module.
\end{abstract}



\section{\label{section:introduction}Introduction}
Let $V$ be a vertex operator algebra and $G$ a finite automorphism group of $V$.
One of the main problems about $V^G$ is to describe the $V^G$-modules in terms of $V$ and $G$.
It is conjectured that under some conditions on $V$,
every simple $V^G$-module is contained in some simple $g$-twisted 
$V$-module for some $g\in G$ (cf. \cite{DVVV}).
Let $M(1)$ be the vertex operator algebra associated to
the Heisenberg algebra of rank $1$ and let $M(1)^{+}$ be the subalgebra of $M(1)$ consisting of the fixed points of an automorphism $\theta$ of $M(1)$ of order $2$ 
(see \eqref{eq:theta}). 
The conjecture above is confirmed for many examples including $M(1)^{+}$ (cf. \cite{AD,DN1,DN2,DN3,TY1,TY2}).
In those examples, they classify the simple $V^{G}$-modules directly
by investigating the Zhu algebra, which is an associative $\C$-algebra introduced in \cite{Z},
since \cite[Theorem 2.2.1]{Z} says that
for a vertex operator algebra $V$
there is a one to one correspondence between the set of all isomorphism classes of simple $\N$-graded weak $V$-modules 
and that of simple modules for the Zhu algebra associated to $V$, where we note that
an arbitrary $V$-module is automatically an $\N$-graded weak $V$-module.

We now turn to non-$\N$-graded weak $V$-modules.
The conjecture makes sense even for non-$\N$-graded weak $V$-modules, however
it has not been confirmed for any example so far since no useful tool like the Zhu algebras is known for these modules.
In this paper we confirm this conjecture for  a class of simple non-$\N$-graded weak $M(1)^{+}$-modules defined by using Whittaker vectors  
for the Virasoro algebra as explained below.

Whittaker modules (Whittaker vectors) are non weight modules defined over various Lie algebras, first appeared in \cite{AP} for $sl_{2}$.
They are defined and studied for all semisimple finite-dimensional complex Lie algebras in \cite{K}, for the Virasoro algebra in \cite{OW} and \cite{LGZ},
and for the affine Kac--Moody algebra $A_1^{(1)}$ in \cite{ALZ}.
Whittaker modules for the Virasoro algebra also appear in the study of 
two-dimensional conformal field theory in physics(cf. \cite{G}, \cite{GT}).
Whittaker modules for a vertex operator algebra $V$ are not defined in general, however, we note
that for the Virasoro element (the conformal vector) $\omega$ of $V$,
$\omega_{n+1}=L(n),n\in\Z$ satisfy the Virasoro algebra relations (cf. \cite[(1.3.4)]{LL}).
Thus, based on the definition of Whittaker vectors for the Virasoro algebra in \cite{LGZ,OW},
it is natural to define Whittaker vectors for $\omega$ 
as follows:
for a weak $V$-module $M$, a non-zero element $\lw$ of $M$ is called a {\it Whittaker vector for $\omega$}  if 
there exists an integer $\ws$ with $\ws\geq 2$ and 
$\ul{\lambda}=
(\lambda_{\lfloor \ws/2\rfloor+1},\lambda_{\lfloor \ws/2\rfloor+2},\ldots,\lambda_{\ws})\in \C^{\ws-\lfloor \ws/2\rfloor}$
with $\lambda_{\ws}\neq 0$ such that
\begin{align}
\omega_{\wi}\lw=
\left\{
\begin{array}{ll}
\lambda_{\wi}\lw,& \wi=\lfloor\ws/2\rfloor +1,\lfloor\ws/2\rfloor +2,\ldots,\ws.\\
0,&\mbox{if  }i>\ws
\end{array}\right.
\end{align}
where $\lfloor\ws/2\rfloor=\max\{i\in\Z\ |\ i\leq \ws/2\}$.
We call $\ul{\lambda}$ the {\it type} of $\lw$. 

The following is the main result of this paper, which implies that
any simple weak $M(1)^{+}$-module with at least one Whittaker vector
is isomorphic to some simple weak $M(1)$-module or to some $\theta$-twisted simple weak $M(1)$-module.
Namely, the conjecture holds for such simple weak $M(1)^{+}$-modules.
\begin{theorem}
\label{theorem:classification-module}
Let $\module$ be a non-zero weak $M(1)^{+}$-module generated by a Whittaker vector for $\omega$.
Then, $M$ is simple. 
The following is a complete set of representatives of equivalence classes of simple weak $M(1)^{+}$-modules with 
at least one Whittaker vector for $\omega$:
\begin{enumerate}
\item
$M(1,\ul{\zeta})\cong M(1,-\ul{\zeta}),\ \ul{\zeta}\in \C^{r}\times \C^{\times}$, $r=1,2,\ldots$.
\item
$M(1,\ul{\zeta})(\theta)\cong M(1,-\ul{\zeta})(\theta),\ \ul{\zeta}\in \C^{r-1}\times \C^{\times}$, $r=1,2,\ldots$.
\end{enumerate}
\end{theorem}
Here, $M(1,\ul{\zeta})$ are simple weak $M(1)$-modules defined in \eqref{eq:untwist-induced} 
and $M(1,\ul{\zeta})(\theta)$ are simple $\theta$-twisted weak $M(1)$-modules defined in \eqref{eq:twist-induced}.
Let us explain the basic idea briefly.
It is shown in \cite[Theorem 2.7 (2)]{DG} that $M(1)^{+}$ is generated by the Virasoro element $\omega$ and homogeneous $J\in M(1)^{+}$ of weight $4$.
For these $\omega$ and $J$, we first find two relations on $M(1)^{+}$ with the help of a computer algebra system Risa/Asir.
Using these relations and the Borcherds identity,
we show that an arbitrary weak $M(1)^{+}$-module $M$ 
generated by a Whittaker vector $\lw$ for $\omega$ 
as in the theorem is generated by $\lw$
 as a module for the Virasoro algebra associated to $\omega$, and the actions of $J$ on $M$ are uniquely
determined by the actions of $\omega$ on $M$.
Thus, $M$ is a Whittaker module for the Virasoro algebra in the sense of \cite{LGZ,OW}, and it follows by \cite[Corollary 4.2]{OW} and \cite[Theorem 7]{LGZ} that $M$ is simple and is uniquely determined by the type 
of a Whittaker vector in $M$. 
Therefore $M$ is isomorphic to one of the weak $M(1)^{+}$-modules listed in the theorem which contains a Whittaker vector of the same type.

The organization of the paper is as follows. 
In Section \ref{section:preliminary} we recall some basic properties of the
vertex operator algebra $M(1)$ associated to the Heisenberg algebra of rank $1$ and its weak modules.
In Section \ref{section:main} we give a proof
of Theorem \ref{theorem:classification-module}.

\section{\label{section:preliminary}Preliminary}
We assume that the reader is familiar with the basic knowledge on
vertex algebras as presented in \cite{B,FLM,LL}. 

Throughout this paper, $\N$ denotes the set of all non-negative integers,
$\C^{\times}=\{z\in\C\ |\ z\neq 0\}$ and
$(V,Y,{\mathbf 1},\omega)$ is a vertex operator algebra.
Recall that $V$ is the underlying vector space, 
$Y(-,\wx)$ is the linear map from $V\otimes_{\C}V$ to $V\db{x}$,
${\mathbf 1}$ is the vacuum vector, and $\omega$ is the Virasoro element.
A weak $V$-module $M$ (cf. \cite[p.157]{Li1}) is called {\it $\N$-graded}
if $M$ admits a decomposition $M=\oplus_{j=0}^{\infty}M(j)$ such that
\begin{align}
\label{eq:N-graded}
a_{k}M(j)\subset M(\wt a+j-k-1)
\end{align}
for homogeneous $a\in V$, $j\in \N$, and $k\in\Z$.
For $i\in\Z$, define
\begin{align}
\Z_{< i}&=\{k\in\Z\ |\ k< i\}\mbox{ and }\Z_{> i}=\{k\in\Z\ |\ k> i\}.
\end{align}

In this section, we recall the vertex operator algebra $M(1)$ associated to the Heisenberg algebra of rank $1$, its automorphism $\theta$ defined in
\eqref{eq:theta}, the subalgebra $M(1)^{+}$ of $M(1)$ consisting of the fixed points of $\theta$,
and ($\theta$-twisted) weak $M(1)$-modules.
Let $\hei$ be a one dimensional vector space equipped with a nondegenerate symmetric bilinear form
$\langle -, -\rangle$. We take $h\in H$ such that $\langle h,h\rangle=1$.
Set a Lie algebra
\begin{align}
\hat{\hei}&=\hei\otimes \C[t,t^{-1}]\oplus \C K
\end{align} 
with the Lie bracket relations 
\begin{align}
[\alpha\otimes t^{m},\beta\otimes t^{n}]&=m\langle\alpha,\beta\rangle\delta_{m+n,0}K,&
[K,\hat{\hei}]&=0
\end{align}
for $\alpha,\beta\in \hei$ and $m,n\in\Z$.
For $\alpha\in H$ and $n\in\Z$, $\alpha(n)$ denotes $\alpha\otimes t^{n}\in\widehat{H}$. 
Set two Lie subalgebras of $\widehat{H}$:
\begin{align}
\widehat{H}_{{\geq 0}}&=\bigoplus_{n\geq 0}H\otimes t^n&\mbox{ and }&&
\widehat{H}_{<0}&=\bigoplus_{n\leq -1}H\otimes t^n.
\end{align}

Let $r$ be a non-negative integer.
For an $r+1$-tuple $\ul{\zeta}=(\zeta_0,\ldots,\zeta_{r})\in\C^{r+1}$,
$\C u_{\ul{\zeta}}$ denotes a one dimensional $\widehat{H}_{{\geq 0}}$-module uniquely determined
by
\begin{align}
h(i)\cdot u_{\zeta}&
=\left\{
\begin{array}{ll}
\zeta_iu_{\zeta}&\mbox{ for }i=0,\ldots,r,\\
0&\mbox{ for }i>r
\end{array}
\right.&&\mbox{ and }&
K\cdot u_{\zeta}&=u_{\zeta}.
\end{align}
We take an $\widehat{H}$-module 
\begin{align}
\label{eq:untwist-induced}
M(1,\zeta)&=\sU (\widehat{H})\otimes_{\sU (\widehat{H}_{\geq 0})}\C u_{\zeta}
\cong \sU(\widehat{H}_{<0})\otimes_{\C}\C u_{\ul{\zeta}}
\end{align}
where $\sU(\fg)$ is the universal enveloping algebra of a Lie algebra $\fg$.
Then, $M(1)=M(1,(0))$ has a vertex operator algebra structure with
the Virasoro element
\begin{align}
\omega&=\dfrac{h(-1)^2\vac}{2}
\end{align}
and $M(1,\ul{\zeta})$ is a simple weak $M(1)$-module for any $\ul{\zeta}\in\C^{r+1}$.
The vertex operator algebra $M(1)$ is called the {vertex operator algebra associated to
 the Heisenberg algebra} $\oplus_{0\neq n\in\Z}H\otimes t^{n}\oplus \C K$. 
If $r=0$, then $M(1,(\lambda_0))$ is a simple $M(1)$-module. 
Since for $i=r+1,r+2,\ldots,2r+1$,
\begin{align}
\label{eq:un-type}
\omega_{i}u_{\ul{\zeta}}&=\dfrac{1}{2}\sum_{\begin{subarray}{l}j,k\in\N\\j+k=i-1\end{subarray}}
h(j)h(k)u_{\ul{\zeta}}=\dfrac{1}{2}\sum_{\begin{subarray}{l}j,k\in\N\\j+k=i-1\end{subarray}}
\zeta_{j}\zeta_{k}u_{\ul{\zeta}}\in\C u_{\ul{\zeta}},
\end{align}
$u_{\ul{\zeta}}$ is a Whittaker vector of type $(\sum_{\begin{subarray}{l}j,k\in\N\\j+k=i-1\end{subarray}}
\zeta_{j}\zeta_{k}/2)_{i=r+1}^{2r+1}$ for $\omega$.
If $r\geq 1$ and $\zeta_{r}\neq 0$, then $\lw$ is an 
eigenvector for $\omega_{2r+1}$ with eigenvalue $\zeta_{r}^2/2$
and hence $M(1,\ul{\zeta})$ is not an $\N$-graded weak $M(1)$-module by \eqref{eq:N-graded}.
We also note that the map 
\begin{align}
\label{eq:untwist-zeta-lambda}
\C^{r}\times \C^{\times}&\rightarrow \C^{r}\times \C^{\times}\nonumber\\
(\zeta_0,\ldots,\zeta_r)&\mapsto (\dfrac{1}{2}\sum_{\begin{subarray}{l}j,k\in\N\\j+k=i-1\end{subarray}}
\zeta_{j}\zeta_{k})_{i=r+1}^{2r+1}
\end{align}
is onto and
the images of $\ul{\zeta}, \ul{\zeta}^{\prime}\in \C^{r}\times \C^{\times}$ under this map
are equal if and only if $\ul{\zeta}=\pm \ul{\zeta}^{\prime}$ since
\begin{align}
&\dfrac{1}{2}\sum_{\begin{subarray}{l}j,k\in\N\\j+k=i-1\end{subarray}}
\zeta_{j}\zeta_{k}\nonumber\\
&=\left\{
\begin{array}{ll}
\dfrac{1}{2}\zeta_{r}^2 &\mbox{if }i=2r+1,\\
\zeta_{r}\zeta_{i-1-r}+\dfrac{1}{2}\sum\limits_{j=i-1-r+1}^{r-1}\zeta_{j}\zeta_{i-1-j} &\mbox{if }r+1\leq i<2r+1,
\end{array}
\right.
\end{align}
for $i=r+1,r+2,\ldots,2r+1$.

Let $\theta$ be an automorphism of $M(1)$ of order $2$ determined by
\begin{align}
\label{eq:theta}
\theta(h(-i_1)\cdots h(-i_n)\vac)&=(-1)^{n}h(-i_1)\cdots h(-i_n)\vac.
\end{align}

Set a Lie algebra
\begin{align}
\hat{\hei}[-1]&=\hei\otimes t^{1/2}\C[t,t^{-1}]\oplus \C K
\end{align} 
with the Lie bracket relations 
\begin{align}
[K,\hat{\hei}[-1]]&=0&\mbox{and}&&
[\alpha\otimes t^{m},\beta\otimes t^{n}]&=m\langle\alpha,\beta\rangle\delta_{m+n,0}K
\end{align}
for $\alpha,\beta\in \hei$ and $m,n\in1/2+\Z$.
For $\alpha\in H$ and $n\in1/2+\Z$, $\alpha(n)$ denotes $\alpha\otimes t^{n}\in\widehat{H}$. 
Set two Lie subalgebras of $\hat{\hei}[-1]$:
\begin{align}
\widehat{H}[-1]_{{>0}}&=\bigoplus_{n\in 1/2+\N}H\otimes t^n&\mbox{ and }&&
\widehat{H}[-1]_{{<0}}&=\bigoplus_{n\in 1/2+\N}H\otimes t^{-n}.
\end{align}
Let $r$ be a positive integer.
For an $r$-tuple $\ul{\zeta}=(\zeta_{1/2},\zeta_{1/2+1},\ldots,\zeta_{r-1/2})\in\C^{r}$,
$\C u_{\ul{\zeta}}$ denotes a unique one dimensional $\widehat{H}[-1]_{{>0}}$-module 
such that 
\begin{align}
h(i)\cdot u_{\ul{\zeta}}&
=\left\{
\begin{array}{ll}
\zeta_iu_{\ul{\zeta}}&\mbox{ for }i=1/2,1/2+1,\ldots,r-1/2,\\
0&\mbox{ for }i>r-1/2
\end{array}
\right.&\mbox{and}\nonumber\\
K\cdot u_{\ul{\zeta}}&=u_{\ul{\zeta}}.
\end{align}
We take an $\widehat{H}[-1]$-module 
\begin{align}
\label{eq:twist-induced}
M(1,\ul{\zeta})(\theta)
&=\sU (\widehat{H}[-1])\otimes_{\sU (\widehat{H}[-1]_{>0})}\C u_{\ul{\zeta}}
\cong\sU (\widehat{H}[-1]_{<0})\otimes_{\C}\C u_{\ul{\zeta}}.
\end{align}
We define for $\alpha\in \hei$, 
\begin{align}
	\alpha(x)&=\sum_{i\in 1/2+\Z}\alpha(i)x^{-i-1}
	\end{align}
and for $u=\alpha_1(-\wi_1)\cdots \alpha_k(-\wi_k)\vac\in M(1)$, 
\begin{align}
	Y_{0}(u,x)&=\nor
\dfrac{1}{(\wi_1-1)!}	(\dfrac{d^{\wi_1-1}}{dx^{\wi_1-1}}\alpha_1(x))
	\cdots
\dfrac{1}{(\wi_k-1)!}	(\dfrac{d^{\wi_k-1}}{dx^{\wi_k-1}}\alpha_k(x))\nor.
\end{align}
Here, for $\beta_1,\ldots,\beta_{n}\in H$ and $i_1,\ldots,i_n\in1/2+\Z$, we define 
$\nor \beta_1(i_1)\cdots\beta_{n}(i_n)\nor$ inductively by
\begin{align}
\label{eq:nomal-ordering}
\nor \beta_1(i_1)\nor&=\beta_1(i_1)\qquad\mbox{ and}\nonumber\\
\nor \beta_1(i_1)\cdots\beta_{n}(i_n)\nor&=
\left\{
\begin{array}{ll}
\nor \beta_{2}(i_2)\cdots\beta_{n}(i_n)\nor \beta_1(i_1)&\mbox{if }i_1\geq 0,\\
\beta_{1}(i_1)\nor \beta_{2}(i_2)\cdots\beta_{n}(i_n)\nor &\mbox{if }i_1<0.
\end{array}\right.
\end{align}
We define $c_{mn}\in\Q$ for $ m,n\in \Z_{\geq 0}$ by
\begin{align}
	\sum_{m,n=0}^{\infty}c_{mn}x^{m}y^{n}&=-\log (\dfrac{(1+x)^{1/2}+(1+y)^{1/2}}{2})
	\end{align}
	and set 
	\begin{align}
		\Delta_{x}&=\sum_{m,n=0}^{\infty}c_{mn}h(m)h(n)x^{-m-n}.
		\end{align}
Then, for $u\in M(1)$ we define a vertex operator $Y_{M(1,\ul{\zeta})(\theta)}$ by
\begin{align}
Y_{M(1,\ul{\zeta})(\theta)}(u,x)&=Y_{0}(e^{\Delta_{x}}u,x).
\end{align}
The same argument as in \cite[Theorem 9.3.1]{FLM} and \cite[Section 4.2]{Li2} shows that 
$(M(1,\ul{\zeta})(\theta),Y_{M(1,\ul{\zeta})(\theta)})$ is a simple $\theta$-twisted weak $M(1)$-module.
If $r=1$ and $\lambda_{1/2}=0$, then $M(1,(0))(\theta)$ is a simple $\theta$-twisted $M(1)$-module.
Since for $i=r+1,r+2,\ldots,2r$,
\begin{align}
\label{eq:twisted-type}
\omega_{i}u_{\ul{\zeta}}&=\dfrac{1}{2}\sum_{\begin{subarray}{l}j,k\in1/2+\N\\j+k=i-1\end{subarray}}
\zeta_{j}\zeta_{k}u_{\ul{\zeta}}\in\C u_{\ul{\zeta}},
\end{align}
$u_{\ul{\zeta}}$ is a Whittaker vector of type $(\sum_{\begin{subarray}{l}j,k\in1/2+\N\\j+k=i-1\end{subarray}}\zeta_{j}\zeta_{k}/2)_{i=r+1}^{2r}$
for $\omega$.
As in the case of $M(1,\ul{\zeta})$,
the map 
\begin{align}
\label{eq:twist-zeta-lambda}
\C^{r-1}\times \C^{\times}&\rightarrow \C^{r-1}\times \C^{\times}\nonumber\\
(\zeta_{1/2},\ldots,\zeta_{r-1/2})&\mapsto (\dfrac{1}{2}\sum_{\begin{subarray}{l}j,k\in1/2+\N\\j+k=i-1\end{subarray}}
\zeta_{j}\zeta_{k})_{i=r+1}^{2r}
\end{align}
is onto and
the images of $\ul{\zeta}, \ul{\zeta}^{\prime}\in \C^{r}\times \C^{\times}$ under this map
are equal if and only if $\ul{\zeta}=\pm \ul{\zeta}^{\prime}$. 

We take the subalgebra $M(1)^{+}$ of $M(1)$ consisting of the fixed points of $\theta$:
\begin{align}
M(1)^{+}&=\{u\in M(1)\ |\ \theta(u)=u\}.
\end{align}
It is shown in \cite[Theorem 2.7 (2)]{DG} that 
$M(1)^{+}$ is generated by the Virasoro element $\omega$
and homogeneous
\begin{align}
J&=h(-1)^4\vac-2h(-3)h(-1)\vac+\dfrac{3}{2}h(-2)^2\vac\in M(1)^{+}
\end{align}
of weight $4$.
It is shown in \cite[(3.3)]{DN1} that $\omega$ and $J$ satisfies
\begin{align}
\label{eq:lie-oj}
[\omega_{i},J_{j}]&=(3i-j)J_{i+j-1},\ i,j\in\Z.
\end{align}
We also have the following commutator formula for $J_{\wi}$ and $J_{\wj}$ ($\wi,\wj\in\Z$) by using a computer algebra system Risa/Asir:
\begin{align}
\label{eq:jj-com}
&[J_{\wi},J_{\wj}]=\sum_{\wk=0}^{7}\binom{\wi}{k}(J_{k}J)_{i+j-k}\nonumber\\
&=
(-\dfrac{1392}{5}\omega_{-6}\vac-\frac{2784}{5}\omega_{-4}\omega_{-1}\vac
+120\omega_{-3}\omega_{-2}\vac+\frac{1632}{5}\omega_{-2}\omega_{-1}^2\vac\nonumber\\
&\qquad{}-\frac{56}{5}\omega_{-2}J_{-1}\vac-\frac{56}{5}\omega_{-1}J_{-2}\vac+\frac{6}{5}J_{-4}\vac)_{{\wi}+{\wj}}\nonumber\\
&\quad{}+\binom{\wi}{1}(-\frac{1856}{5}\omega_{-5}\vac-\frac{2384}{5}\omega_{-3}\omega_{-1}\vac+
\frac{1316}{5}\omega_{-2}^2\vac+\frac{1088}{5}\omega_{-1}^3\vac\nonumber\\
&\qquad{}-\frac{112}{5}\omega_{-1}J_{-1}\vac-\frac{46}{5}J_{-3}\vac)_{{\wi}+{\wj}-1}\nonumber\\
&\quad{}+\binom{\wi}{2}(-48\omega_{-4}\vac+336\omega_{-2}\omega_{-1}\vac-30J_{-2}\vac)_{{\wi}+{\wj}-2}\nonumber\\
&\quad{}+\binom{\wi}{3}(-72\omega_{-3}\vac+336\omega_{-1}^2\vac-60J_{-1}\vac)_{{\wi}+{\wj}-3}\nonumber\\
&\quad{}+216\binom{\wi}{4}(\omega_{-2}\vac)_{{\wi}+{\wj}-4}+432\binom{\wi}{5}(\omega_{-1}\vac)_{{\wi}+{\wj}-5}
+54\binom{\wi}{7}\vac_{{\wi}+{\wj}-7}.
\end{align}
For $u_{\ul{\zeta}}\in M(1,\ul{\zeta})$ (resp. $M(1,\ul{\zeta})(\theta)$), we have
\begin{align}
J_{i}u_{\ul{\zeta}}&=\sum_{\begin{subarray}{l}i_1,i_2,i_3,i_4\in\N\ (\mmbox{resp. }1/2+\N)\\i_1+i_2+i_3+i_4=i-3\end{subarray}}
\zeta_{i_1}
\zeta_{i_2}
\zeta_{i_3}
\zeta_{i_4}u_{\ul{\zeta}}\in \C u_{\ul{\zeta}}
\end{align}
for $i=3r+3,3r+4,\ldots,4r+3$ (resp. $i=3r-1,3r+4,\ldots,4r-2)$ and
$J_{i}\lw=0$ for $i>4r+3$ (resp. $4r-2)$. 

\begin{lemma}
\label{lemma:simple-hei}
For $\ul{\zeta}\in\C^{r}\times \C^{\times}$, $M(1,\ul{\zeta})$ is a simple  weak $M(1)^{+}$-module
and for $\ul{\zeta}\in\C^{r-1}\times \C^{\times}$, $M(1,\ul{\zeta})(\theta)$ is a simple weak $M(1)^{+}$-module.
In particular, for any $\ul{\lambda}\in\C^{\ws-\lfloor \ws/2\rfloor-1}\times \C^{\times}$, $\ws\geq 2$,
there exists a weak $M(1)^{+}$-module with a Whittaker vector of type $\ul{\lambda}$ for $\omega$.
\end{lemma}
\begin{proof}
We only show that $M(1,\ul{\zeta})$ is simple.
The same argument shows that $M(1,\ul{\zeta})(\theta)$ is simple.
The last statement follows from \eqref{eq:un-type},\eqref{eq:untwist-zeta-lambda},
\eqref{eq:twisted-type}, and \eqref{eq:twist-zeta-lambda}.
Note that by \eqref{eq:untwist-induced},
$M(1,\ul{\zeta})$ is spanned by 
\begin{align}
h(-i_1)\cdots h(-i_{n})u_{\zeta}, n\in\N, i_1,\ldots,i_n\in\Z_{>0}.
\end{align}
For $p,q\in\N$ and $n\in\Z$, we have
\begin{align}
\label{eq:hphq}
(h(-p-1)h(-q-1)\vac)_{n+1}&=\sum_{
\begin{subarray}{l}i,j\in\Z,\\
i+j=n
\end{subarray}
}\binom{-i-1}{p}\binom{-j-1}{q}\nor h(i)h(j)\nor
\end{align}
where $\nor h(i)h(j)\nor$ is defined to be
\begin{align}
\nor h(i)h(j)\nor&=
\left\{
\begin{array}{ll}
h(i)h(j)&\mbox{if }j\geq 0,\\
h(j)h(i)&\mbox{if }j<0.
\end{array}
\right.
\end{align}
Let $u\in M(1,\ul{\zeta})$ and let $\wm\in\N$ such that $h(i)u=0$ for all $i>\wm$.
Then, \eqref{eq:hphq} can be written as
\begin{align}
&(h(-p-1)h(-q-1)\vac)_{n+1}u\nonumber\\
&=\sum_{
\begin{subarray}{l}i,j\geq 0,\\
i+j=2\wm-\wn
\end{subarray}
}\binom{i-m-1}{p}\binom{j-m-1}{q}\nor h(-i+m)h(-j+m)\nor u.
\end{align}
Let $n\in\N$ and set $S=\{(i,j)\in\N^2\ |\ i+j\leq \wn\}$.
Since the square matrix $(\binom{i-m-1}{p}\binom{j-m-1}{q})_{(i,j),(p,q)\in S}$ is non-singular by 
Lemma \ref{lemma:determinant} below,
for any pair of $i,j\in\Z$ with $i\leq j$, $h(i)h(j)u$ is a linear combination of 
$(h(-p-1)h(-q-1)\vac)_{n}u$, $p,q\in\N$, $n\in\Z$.
Since for all $i,j,k\in\Z_{>0}$ with $i\neq j$, 
\begin{align}
[h(i)^2,h(-i)^{k}]&=i^2k(k-1)h(-i)^{k-2}+2ikh(-i)^{k-1}h(i),\nonumber\\
[h(i)^2,h(-j)^{k}]&=0,\nonumber\\
[h(i)^2,h(-i)h(-j)]&=2ih(-j)h(i),
\end{align}
and
\begin{align}
&[h(i)h(j),h(-i)h(-j)^{k}]\nonumber\\
&=ijkh(-j)^{k-1}+jkh(-i)h(-j)^{k-1}h(i)+ih(-j)^{k}h(j),
\end{align}
an inductive argument shows that $M(1,\ul{\zeta})$ is simple.
\end{proof}

\begin{remark}
Lemma \ref{lemma:simple-hei} also follows from Theorem\ref{theorem:classification-module}.
\end{remark}

The following result says that the matrix in the proof of Lemma \ref{lemma:simple-hei} above is non-singular.
\begin{lemma}
\label{lemma:determinant} 
Let $n\in\N$ and let $x_0,\ldots,x_n,y_0,\ldots,y_n$ be indeterminants.
Set $S=\{(i,j)\in\N\ |\ i+j\leq n\}$ and 
a square matrix $A=(\binom{x_i}{k}\binom{y_j}{l})_{(i,j),(k,l)\in S}$ of size $(n+1)(n+2)/2$.
Then
\begin{align}
\label{eq:determinant}
\det A&=
\dfrac{
\prod\limits_{0\leq i<i^{\prime}\leq n}(x_i-x_{i^{\prime}})^{n+1-i^{\prime}}
\prod\limits_{0\leq j<j^{\prime}\leq n}(y_j-y_{j^{\prime}})^{n+1-j^{\prime}}}
{
\prod\limits_{0\leq i<i^{\prime}\leq n}(i-{i^{\prime}})^{n+1-i^{\prime}}
\prod\limits_{0\leq j<j^{\prime}\leq n}(j-{j^{\prime}})^{n+1-j^{\prime}}}.
\end{align}
\end{lemma}
\begin{proof}
Let $i,i^{\prime}\in\N$ with $i<i^{\prime}\leq n$.
Since $j=0,1,\ldots,n-i^{\prime}$ satisfies $0\leq i+j\leq n$ and $0\leq i^{\prime}+j\leq n$,
$(x_i-x_{i^{\prime}})^{n+1-i^{\prime}}$ is a factor of $\det A$.
By the same reason, $(y_j-x_{j^{\prime}})^{n+1-j^{\prime}}$ is a factor of $\det A$
for any pair of $j,j^{\prime}\in\N$ with $j<j^{\prime}\leq n$. 
Thus, $\det A$ is divisible by
\begin{align}
\label{eq:factor}
\prod_{0\leq i<i^{\prime}\leq n}(x_i-x_{i^{\prime}})^{n+1-i^{\prime}}
\prod_{0\leq j<j^{\prime}\leq n}(y_j-y_{j^{\prime}})^{n+1-j^{\prime}}.
\end{align}
Since the degrees of $\det A$ and \eqref{eq:factor} are equal to $n(n+1)(n+2)/3$,
 $\det A$ is a scalar multiple of \eqref{eq:factor}.
Since the matrix obtained by substituting $x_i=i$ and $y_i=i$, $i=0,1,\ldots n$, in $A$ 
is an upper triangular matrix with all diagonal elements $1$,
we have \eqref{eq:determinant}.
\end{proof}
For $i\in\Z$, let $M(1)^{+}_i$ be the subspace of $M(1)^{+}$ consisting of the elements with weight $i$.
Since the dimensions of $M(1)^{+}_i$ are obtained by \cite[Theorem 2.7]{DG},
we know there are relations for $\omega$ and $J$ in $M(1)^{+}_9$ and $M(1)^{+}_{10}$.
We have the following two relations $P^{(9)}$ of weight $9$ and $P^{(10)}$ of weight $10$ 
 for $\omega$ and $J$ in $M(1)^{+}$ by using a computer algebra system Risa/Asir.
\begin{lemma}
The following two elements of $M(1)^{+}$ are zero:
\begin{align*}
\sv^{(9)}&=30J_{-6}\vac-30\omega_{-1}J_{-4}\vac+27\omega_{-2}J_{-3}\vac-39\omega_{-3}J_{-2}\vac\\
&\quad{}+16\omega_{-1}^2J_{-2}\vac+52\omega_{-4}J_{-1}\vac-32\omega_{-2}\omega_{-1}J_{-1}\vac&\mbox{and}\\
\sv^{(10)}&=\dfrac{8192}{525}\omega_{-1}^5\vac-\dfrac{2048}{525}\omega_{-1}^3J_{-1}\vac \\
&\quad{}+J_{-2}^2\vac-\dfrac{13856}{105}\omega_{-2}^{2}\omega_{-1}^2\vac-\dfrac{22528}{105}\omega_{-3}\omega_{-1}^3\vac\\&
\quad{}-\dfrac{45624}{175}\omega_{-3}\omega_{-2}^2\vac
-\dfrac{2304}{175} \omega_{-3}^2\omega_{-1}\vac
-\dfrac{134224}{525}\omega_{-4}\omega_{-2}\omega_{-1}\vac \\&
\quad{}-\dfrac{60848}{525}\omega_{-4}^2\vac 
-\dfrac{2176}{75}\omega_{-5}\omega_{-1}^2\vac 
-\dfrac{576}{175}\omega_{-5}\omega_{-3}\vac \\&
\quad{}+\dfrac{117664}{175}\omega_{-6}\omega_{-2}\vac 
+\dfrac{436416}{175}\omega_{-7}\omega_{-1}\vac 
+\dfrac{252832}{175}\omega_{-9}\vac \\&
\quad{}+\dfrac{24184}{1575}\omega_{-2}^2J_{-1}\vac
+\dfrac{65024}{1575}\omega_{-3}\omega_{-1}J_{-1}\vac 
-\dfrac{150176}{1575}\omega_{-5}J_{-1}\vac \\&
\quad{}+\dfrac{152}{525}\omega_{-2}\omega_{-1}J_{-2}\vac
+\dfrac{17102}{1575} \omega_{-4}J_{-2}\vac
+\dfrac{1024}{315}\omega_{-1}^2J_{-3}\vac \\&
\quad{}+\dfrac{2544}{175}\omega_{-3}J_{-3}\vac 
+\dfrac{382}{525}\omega_{-2}J_{-4}\vac 
-\dfrac{1088}{525}\omega_{-1}J_{-5}\vac. 
\end{align*}
\end{lemma}

\section{\label{section:main}Weak $M(1)^{+}$-modules with Whittaker vectors.}
In this section, we will show Theorem \ref{theorem:classification-module}.
Thus, any simple $M(1)^{+}$-module with at least one 
Whittaker vector for $\omega$ is isomorphic to one of the simple $M(1)^{+}$-modules listed in
Lemma \ref{lemma:simple-hei}.
For a weak $V$-module $M$ and $\lw\in M$, 
$\langle \omega\rangle\lw$ denotes the set of linear span of the following elements:
\begin{align}
\omega_{i_1}\cdots \omega_{i_n}w\ (n\in\N, i_1,\ldots,i_n\in\Z).
\end{align}

By \eqref{eq:lie-oj}, an inductive argument shows the following result.
\begin{lemma}
\label{lemma:J-decrease}
Let $M$ be a weak $M(1)^{+}$-module and $\lw\in M$.
Let $\wpp\in\N, \ws_1,\ldots,\ws_{\wpp}\in \Z_{>0}, j\in\Z$, and $\wi_1,\ldots,\wi_{\wpp}\in\N$
such that $\wi_k\leq \ws_k$ for all $k=1,\ldots,\wpp$.
Set
$e=\sum_{\wk=1}^{\wpp}(\ws_{\wk}-i_{\wk})$. Then
\begin{align}
J_{j+e}\omega_{i_p}\cdots\omega_{i_1}\lw&\in \sum_{i\geq j}\langle\omega \rangle J_{i}\lw.
\end{align}
\end{lemma}

The following result will be used to compute $P^{(9)}_{i}\lw$ and $P^{(10)}_{i}\lw, i\in\Z$, for a Whittaker vector $\lw$ in a weak $M(1)^{+}$-module
in the proof of Lemma \ref{lemma:gen-omega}.
\begin{lemma}
\label{lemma:top}
Let $\ws, \wot\in\Z$ with $2\leq \wot\leq \ws$.
Let $M$ be a weak $M(1)^{+}$-module and $\lw\in M$ such that
$\omega_{i}\lw=0$ for all $i>\ws$.
\begin{enumerate}
\item
Let
$p,q,i_1,\ldots,i_{p},i_{p+1},\ldots,i_{q}\in\Z$
such that $0\leq p\leq q$, $i_1,\ldots,i_{p}\geq 0,i_{p+1},\ldots,i_{q}<0$, and
	$\ws \wq\leq \wi_1+\cdots+\wi_{q}$. If $p<q$ or 
there exists $l\in\{1,\ldots,p\}$ such that $i_{l}\neq \ws$, then
	\begin{align}
\label{eq:oo}
		\omega_{i_q}\cdots \omega_{i_{p+1}}\omega_{i_p}\cdots \omega_{i_1}\lw&=0.
	\end{align}
	\item
	Let
 $p,q,i_1,\ldots,i_{p},i_{p+1},\ldots,i_{q},n\in\Z$
such that $0\leq p\leq q$, $i_1,\ldots,i_{p}\geq 0,i_{p+1},\ldots,i_{q}<0$ and
set 
\begin{align}
\label{eq:def-j}
\wj&=\wn-\wi_1-\cdots-\wi_{q}+\ws (q-1)+\wot.
\end{align}
If $p<q$ or 
there exists $l\in\{1,\ldots,p\}$ such that $i_{l}\not\in\{\wot,\wot+1,\ldots \ws\}$, then
	\begin{align}
\label{eq:ojo}
		\omega_{i_q}\cdots \omega_{i_{p+1}}J_{\wj}\omega_{i_p}\cdots \omega_{i_1}\lw&\in \sum_{i\geq n+1}\langle\omega \rangle J_{i}\lw.
	\end{align}
\end{enumerate}
\end{lemma}
\begin{proof}
We will show (2) by induction on $p$. The same argument shows (1).
If $p=0$ and $0=p<q$, then $\wi_1,\ldots,\wi_q<0$ and hence
$\wj>\wn$ by \eqref{eq:def-j}. Thus, $\eqref{eq:ojo}$ holds.
Let $p=1$. If $\wi_1>\ws$, then $\eqref{eq:ojo}$ clearly holds.
Assume $\wi_1\leq \ws$ and $1=p<q$.
Then, $\wi_2,\ldots,\wi_q<0$ and therefore
\begin{align}
\label{eq:wigeq}
\wj&=(\ws(q-1)-\wi_1)+\wot+\wn-\wi_2-\cdots-\wi_{q}>\wn&\mbox{and}\nonumber\\
\wi_1+\wj-1&=\ws(q-1)+\wot+\wn-\wi_2-\cdots-\wi_{q}-1>\wn.
\end{align}
By \eqref{eq:lie-oj}, we have
\begin{align}
\omega_{i_q}\cdots \omega_{i_{2}}J_{\wj}\omega_{i_1}\lw&=
\omega_{i_q}\cdots \omega_{i_{2}}((-3\wi_1+\wj)J_{\wi_1+\wj-1}+\omega_{i_1}J_{\wj})\lw\nonumber\\
&\in \sum_{i\geq n+1}\langle\omega \rangle J_{i}\lw.
\end{align}
Suppose that $p=q=1$ and $\wi_1<t$. Then $\wj= \wot+\wn-\wi_1>\wn$ and
$\wi_1+\wj-1=\wot+\wn-1>\wn$. The same argument as above shows that $\eqref{eq:ojo}$ holds.

Let $p>1$. Suppose that there exists $l\in\{1,\ldots,p\}$ such that $i_{l}>\ws$.
Since
\begin{align}
\label{eq:induction-omega}
&\omega_{i_{l}}\cdots \omega_{i_1}\lw\nonumber\\
&=\sum_{k=1}^{l-1}\omega_{i_{l-1}}\cdots \omega_{i_{k+1}}[\omega_{i_{l}},\omega_{i_{k}}]
\omega_{i_{k-1}}\cdots\omega_{i_{1}}\lw+\omega_{i_{p-1}}\cdots \omega_{i_1}\omega_{i_{l}}\lw\nonumber\\
&=\sum_{k=1}^{l-1}(i_{l}-i_{k})\omega_{i_{l-1}}\cdots \omega_{i_{k+1}}\omega_{i_{l}+i_{k}-1}
\omega_{i_{k-1}}\cdots\omega_{i_{1}}\lw,
\end{align}
it is sufficient to show that
\begin{align}
\label{eq:ojo-proof}
&\omega_{i_q}\cdots \omega_{i_{p+1}}J_{\wj}\omega_{i_p}\cdots \omega_{i_{l+1}}\omega_{i_{l-1}}\cdots
\omega_{i_{k+1}}\omega_{i_{l}+i_{k}-1}\omega_{i_{k-1}}\cdots\omega_{i_{1}}\lw\nonumber\\
&\in \sum_{i\geq n+1}\langle\omega \rangle J_{i}\lw
\end{align}
for all $k\in\{1,2,\ldots,l\}$.
Let $k\in\{1,2,\ldots,l\}$. Since
\begin{align}
\label{eq:il}
i_{l}+i_{k}-1> \ws+i_{k}-1>i_{k}\geq 0
\end{align}
and
\begin{align}
\wj&=(\wn+\ws-1)-\sum_{
\begin{subarray}{l}1\leq m\leq\wq\\
m\neq k, l
\end{subarray}
}-i_{m}-(i_{l}+i_{k}-1)+\ws(q-2)+t,
\end{align}
if one of $i_1,\ldots,i_{k-1},i_{l}+i_{k}-1,i_{k+1},\ldots,i_{p}$
is greater than $\ws$, then the left-hand side of $\eqref{eq:ojo-proof}$ is an element of
$\sum_{i\geq n+\ws}\langle\omega \rangle J_{i}\lw\subset \sum_{i\geq n+1}\langle\omega \rangle J_{i}\lw$ by the induction hypothesis.
Suppose that $i_1,\ldots,i_{k-1},i_{l}+i_{k}-1,i_{k+1},\ldots,i_{p}$ are all at most $\ws$.
Since $j$ can be written as 
\begin{align}
\wj&=(\wn+1)+\sum_{
\begin{subarray}{l}1\leq m\leq\wpp\\
m\neq k, l
\end{subarray}
}(\ws-i_{m})+(\ws-(i_{l}+i_{k}-1))\nonumber\\
&\quad{}-\sum_{m=p+1}^{q}i_{m}+s(q-p)+t-2,
\end{align}
$\eqref{eq:ojo-proof}$ holds by Lemma \ref{lemma:J-decrease}.

Suppose that $i_1,\ldots,i_{p}\leq \ws$. If $p<q$, then since
\begin{align}
\label{eq:case-1}
\wj&=(\wn+1)+\sum_{k=1}^{\wpp}(s-i_k)-\sum_{k=p+1}^{q}i_{k}+s(q-p-1)+t-1,
\end{align}
$\eqref{eq:ojo}$ holds by Lemma \ref{lemma:J-decrease}.
If $p=q$ and there exists $\wl\in\{1,\ldots,\wpp\}$ such that $\wi_{\wl}<\wot$, then since
\begin{align}
\label{eq:case-2}
\wj&=(\wn+1)+\sum_{1\leq k\leq \wpp,k\neq l}(s-i_k)+(t-1-\wi_{l}),
\end{align}
$\eqref{eq:ojo}$ holds by Lemma \ref{lemma:J-decrease}.
\end{proof}

For $\wpp_1,\wpp_2,\wpp_3\in\Z_{>0}$ and $n\in\Z$, by using \cite[(3.8.9)]{LL} we have
\begin{align}
\label{eq:oj}
(\omega_{-\wpp_1}J_{-\wpp_2}\vac)_{\wn}
&=\sum_{
\begin{subarray}{l}\wi<0,\wj\in\Z\\
\wi+\wj=\wn+1-\wpp_1-\wpp_2\end{subarray}}
\binom{-i-1}{\wpp_1-1}
\binom{-\wj-1}{\wpp_2-1}
\omega_{i}J_{\wj} \nonumber\\
&\quad{}+\sum_{\begin{subarray}{l}\wi\geq 0,\wj\in\Z\\
i+\wj=\wn+1-\wpp_1-\wpp_2\end{subarray}}\binom{-i-1}{\wpp_1-1}\binom{-\wj-1}{\wpp_2-1}
J_{\wj}\omega_{\wi}
\end{align}
and
\begin{align}
\label{eq:ooj}
&(\omega_{-\wpp_1}\omega_{-\wpp_2}J_{-\wpp_3}\vac)_{n}\nonumber\\
&=\sum_{
\begin{subarray}{l}i_1<0,i_2<0,\wj\in\Z\\
i_1+i_2+\wj=n+1-\wpp_1-\wpp_2-\wpp_3\end{subarray}}
\binom{-i_1-1}{\wpp_1-1}
\binom{-i_2-1}{\wpp_2-1}
\binom{-\wj-1}{\wpp_3-1}
\omega_{i_1}\omega_{i_2}J_{\wj} \nonumber\\
&\quad{}+
\sum_{\begin{subarray}{l}i_1<0,i_2\geq 0,\wj\in\Z
\\i_1+i_2+\wj=n+1-\wpp_1-\wpp_2-\wpp_3\end{subarray}}\binom{-i_1-1}{\wpp_1-1}
\binom{-i_2-1}{\wpp_2-1}
\binom{-\wj-1}{\wpp_3-1}
\omega_{i_1}J_{\wj}\omega_{i_2}\nonumber\\
&\quad{}+
\sum_{\begin{subarray}{l}i_1\geq 0,i_2<0,\wj\in\Z\\i_1+i_2+\wj=n+1-\wpp_1-\wpp_2-\wpp_3\end{subarray}}
\binom{-i_1-1}{\wpp_1-1}\binom{-i_2-1}{\wpp_2-1}
\binom{-\wj-1}{\wpp_3-1}
\omega_{i_2}J_{\wj}\omega_{i_1} \nonumber\\
&\quad{}+\sum_{\begin{subarray}{l}i_1\geq 0,i_2\geq 0,\wj\in\Z
\\i_1+i_2+\wj=n+1-\wpp_1-\wpp_2-\wpp_3\end{subarray}}\binom{-i_1-1}{\wpp_1-1}\binom{-i_2-1}{\wpp_2-1}
\binom{-\wj-1}{\wpp_3-1}
J_{\wj}\omega_{i_2}\omega_{i_1}.
\end{align}

The following  is a key result to show Theorem \ref{theorem:classification-module}.
\begin{lemma}
\label{lemma:gen-omega}
Let $\ws, \wot\in\Z$ such that $2\leq \wot\leq \ws$
and let $\ul{\lambda}=(\lambda_{t},\ldots,\lambda_{\ws})\in \C^{\ws-\wot+1}$
with $\lambda_{\ws}\neq 0$.
Let $M$ be a weak $M(1)^{+}$-module and $\lw\in M$ such that
$\omega_{i}\lw=\lambda_{i}\lw$ 
for $i=\wot,\ldots,\ws$ and 
$\omega_{i}\lw=0$ for all $i>\ws$. Then 
for an arbitrary $i\in\Z$, $J_{i}\lw$ is an element of $\langle \omega\rangle \lw$ and therefore
the submodule of $M$ generated by $\lw$ is equal to $\langle \omega\rangle \lw$.
Moreover for $i\in\Z$, we have
\begin{align}
\label{eq:Jw}
J_{i}\lw&=\left\{
\begin{array}{ll}
4\lambda_{\ws}^2\lw,& \mbox{if }\wi=2\ws+1,\\
0,&\mbox{if }\wi>2\ws+1,\mbox{ and }\\
\mu_i\lw & \mbox{if }\ws+\wot+1\leq \wi\leq 2\ws+1
\end{array}\right.
\end{align}
where $\mu_i$ is a polynomial in $\lambda_{i-\ws-1},\ldots,\lambda_{\ws}$ for each $i=\ws+\wot+1,\ws+\wot+2,\ldots,2\ws+1$.
\end{lemma}
\begin{proof}
Let $n\in\Z$. Applying Lemma \ref{lemma:J-decrease} with $p=1$, $s_1=s$, and $j=n+\ws-1$, we have
\begin{align}
J_{n+2\ws-1-\wi}\omega_{\wi}\lw &
\in \sum_{j\geq n+1}\langle \omega \rangle J_{j}\lw
\end{align}
for $i=0,1,\ldots,\ws$.
Thus, for positive integers $p_1$ and $p_2$ with $p_1+p_2=5$, it follows from \eqref{eq:oj} that
\begin{align}
\label{eq:oj-induction-1}
(\omega_{-\wpp_1}J_{-\wpp_2}\vac)_{n+2\ws+3}\lw&\in \sum_{j\geq n+1}\langle \omega \rangle J_{j}\lw.
\end{align}
For $i_1\in\N$ and $i_2\in\Z_{<0}$, applying Lemma \ref{lemma:top} (2) with $\wq=2, \wpp=1$ and $\wot=\ws$, we have
\begin{align}
\label{eq:oj-induction-2}
\omega_{\wi_2}J_{n+2\ws-\wi_1-\wi_2}\omega_{\wi_1}\lw &
\in \sum_{j\geq n+1}\langle \omega \rangle J_{j}\lw.
\end{align}
For $i_1,i_2\in\N$ such that at least one of $i_1,i_2$ is not $\ws$, applying Lemma \ref{lemma:top} (2) with $\wq=\wpp=2$ and $\wot=\ws$, we have
\begin{align}
\label{eq:oj-induction-3}
J_{n+2\ws-\wi_1-\wi_2}\omega_{\wi_2}\omega_{\wi_1}\lw &
\in \sum_{j\geq n+1}\langle \omega \rangle J_{j}\lw.
\end{align}
Thus, by \eqref{eq:ooj} we have
\begin{align}
\label{eq:oj-induction-4}
&(\omega_{-1}^2J_{-2}\vac-2\omega_{-2}\omega_{-1}J)_{\wn+2s+3}w\nonumber\\
& \equiv
(-\wn-1)J_{\wn}\omega_{s}^2\lw-2(-\ws-1)J_{\wn}\omega_{s}^2\lw\pmod{\sum_{j\geq n+1}\langle \omega \rangle J_{j}\lw}\nonumber\\
&=
(-\wn+2s+1)\lambda_{s}^2J_{\wn}w.
\end{align}
It follows from \eqref{eq:oj-induction-1} and \eqref{eq:oj-induction-4} that
\begin{align}
\label{eq:J-ind}
0&=\dfrac{1}{16}\sv^{(9)}_{\wn+2s+3}\lw\equiv (-\wn+2s+1)\lambda_{s}^2J_{\wn}w \pmod{\sum_{j\geq n+1}\langle \omega \rangle J_{j}\lw}.
\end{align}
Since $M(1)^{+}$ is simple by \cite[Theorem 4.4]{DM}, it follows from \cite[Proposition 11.9]{DL} that
there exists $m\in\Z$ such that $J_{m}\lw\neq 0$ and $J_{\wj}\lw=0$ for all $\wj>m$.
By \eqref{eq:J-ind} we have $m=2s+1$.
By \eqref{eq:lie-oj} and \eqref{eq:J-ind}, an inductive argument shows
\begin{align}
\label{eq:Jinomega}
J_{n}\lw\in \langle \omega \rangle J_{2\ws+1}\lw
\end{align}
for all $n\in\Z$.
Applying the same argument as above, by \eqref{eq:jj-com} and Lemma \ref{lemma:top} we have
\begin{align}
&(J_{-2}^2\vac)_{5\ws+4}\lw
=
\sum_{
\begin{subarray}{l}
i<0,j\in\Z\\
i+j=5\ws+1
\end{subarray}
}(-i-1)(-j-1)J_{i}J_{j}\lw+\sum_{
\begin{subarray}{l}
i\geq 0,j\in\Z\\
i+j=5\ws+1
\end{subarray}
}(-i-1)(-j-1)J_{j}J_{i}\lw\nonumber\\
&=\sum_{i=0}^{2\ws+1}
\sum_{
\begin{subarray}{l}
j\in\Z\\
i+j=5\ws+1
\end{subarray}
}(-i-1)(-j-1)([J_{j},J_{i}]+J_{i}J_{j})\lw\nonumber\\
&=0
\end{align}
and 
\begin{align}
0&=P^{(10)}_{5s+4}w
=(\dfrac{8192}{525}\omega_{-1}^{5}\vac-\dfrac{2048}{525}\omega_{-1}^3J_{-1}\vac)_{5s+4}\lw
=\dfrac{2048}{525}(4\lambda_{s}^{5}-\lambda_{s}^3J_{2s+1})\lw.
\end{align}
Thus
\begin{align}
J_{2s+1}\lw&=4\lambda_{s}^2\lw
\end{align}
and therefore $J_{n}\lw\in \langle \omega \rangle \lw$ for all $n\in\Z$ by \eqref{eq:Jinomega}.

Applying the same argument as above, by Lemma \ref{lemma:top} (2) we have
\begin{align}
0&=\dfrac{1}{16}P^{(9)}_{3s+\wot+4}\lw=(\omega_{-1}^2J_{-2}\vac-2\omega_{-2}\omega_{-1}J)_{3\ws+\wot+4}w\nonumber\\
&=\sum_{\wj\in\Z}
\sum_{
\begin{subarray}{l}\wot\leq \wi_1,\wi_2\leq \ws,\\\wi_1+\wi_2+\wj=\ws+\wot+(2\ws+1)\end{subarray}}
(-\wj+2\wi_1+1)J_{\wj}\omega_{i_1}\omega_{i_2}\lw\nonumber\\
\nonumber\\
&=\sum_{\wj=\ws+\wot+1}^{2\ws+1}
\sum_{
\begin{subarray}{l}\wot\leq \wi_1,\wi_2\leq \ws,\\\wi_1+\wi_2+\wj=3\ws+\wot+1\end{subarray}}
(-\wj+2\wi_1+1)J_{\wj}\lambda_{\wi_1}\lambda_{\wi_2}\lw\nonumber\\
\nonumber\\
&=(\ws-\wot)J_{\ws+\wot+1}\lambda_{\ws}^2\lw+
\sum_{\wj=\ws+\wot+2}^{2\ws+1}
\sum_{
\begin{subarray}{l}\wot\leq \wi_1,\wi_2\leq \ws,\\\wi_1+\wi_2+\wj=3\ws+\wot+1\end{subarray}}
(-\wj+2\wi_1+1)J_{\wj}\lambda_{\wi_1}\lambda_{\wi_2}\lw.
\end{align}
An inductive argument on $\wot=\ws,\ws-1,\ldots$ shows that  for each $\wi\in\{\ws+\wot+1,\ws+\wot+2\ldots, 2\ws+1\}$, $\mu_i$ is a polynomial in 
$\lambda_{i-\ws-1},\ldots,\lambda_{\ws}$.
\end{proof}

Now we give a proof of Theorem \ref{theorem:classification-module}.

\begin{proof}[(Proof of Theorem \ref{theorem:classification-module})]
Let $\ws\in\Z$ with $\ws\geq 2$ and 
$\lambda=(\lambda_{\lfloor\ws/2\rfloor+1},\lambda_{\lfloor\ws/2\rfloor+2},\ldots,\lambda_{\ws})
\in \C^{\ws-\lfloor\ws/2\rfloor-1}\times\C^{\times}$. 
Taking a quotient space of the tensor algebra of $M(1)^{+}\otimes \C[t,t^{-1}]$ by the two sided ideal
generated by the Borcherds identity, 
$J \otimes t^i\ (i>2\ws+1)$, $\omega\otimes t^{i}\ (i>\ws)$,
and $\omega\otimes t^{i}-\lambda_i(\vac \otimes 1)\ (\lfloor\ws/2\rfloor+1\leq i\leq \ws)$, we obtain a pair $(N,\ulw)$ of a weak $M(1)^{+}$-module $N$ and a Whittaker vector $\ulw\in N$ of type $\ul{\lambda}$ such that
$Y_{N}(J,x)\ulw\in x^{-2\ws-2}N[[x]]$ with the following universal property:
for any pair $(U,u)$ of a weak $M(1)^{+}$-module $U$ and a Whittaker vector $u\in U$ of type $\ul{\lambda}$
such that $Y_{U}(J,x)u\in x^{-2\ws-2}U[[x]]$,
there exists a unique weak $M(1)^{+}$-module homomorphism $N\rightarrow U$ which maps $\ulw$ to $u$.
By Lemmas \ref{lemma:simple-hei} and \ref{lemma:gen-omega}, we have
$N=\langle\omega\rangle\ulw\neq 0$, namely, $N$ is a Whittaker module for the Virasoro algebra
in the sense of \cite{LGZ,OW}, and therefore
$N$ is a simple module for the Virasoro algebra by \cite[Corollary 4.2]{OW} and \cite[Theorem 7]{LGZ}. 
In particular, $N$ is a simple weak $M(1)^{+}$-module.

Let $\module$ be a non-zero weak $M(1)^{+}$-module generated by a Whittaker vector $\lw$
of type $\lambda$ for $\omega$.
By \eqref{eq:Jw}, $\module$ is isomorphic to a quotient weak module of $N$ and, moreover, $N\cong M$ since $N$ is simple.
Thus, $M$ is simple and by the last statement in Lemma \ref{lemma:simple-hei}, $M$ is isomorphic to one of the weak $M(1)^{+}$-modules listed in (1) and (2).
Since any Whittaker vector
of type $\lambda$ for $\omega$ in $M$ is a non-zero scalar multiple of  $\lw$ by \cite[Proposition 3.2]{OW} and \cite[Theorem 2.3]{FJK},
the proof is complete.
\end{proof}

\end{document}